\documentclass[11pt,a4paper,reqno]{amsart}
\usepackage{mathrsfs}
\usepackage{syntonly}
\usepackage{amsmath}
\usepackage{amsthm}
\usepackage{amsfonts}
\usepackage{amssymb}
\usepackage{latexsym}
\usepackage{amscd,amssymb,amsopn,amsmath,amsthm,graphics,amsfonts,mathrsfs,accents,enumerate,verbatim,calc}
\usepackage[dvips]{graphicx}
\usepackage[colorlinks=true,linkcolor=blue,citecolor=blue]{hyperref}
\usepackage[all]{xy}

\date{}
\date{}
\allowdisplaybreaks[4] \footskip=15pt
\renewcommand{\uppercasenonmath}[1]{}

\numberwithin{equation}{section} \theoremstyle{plain}
\newtheorem{lem}{Lemma}[section]
\newtheorem{cor}[lem]{Corollary}
\newtheorem{prop}[lem]{Proposition}
\newtheorem{thm}[lem]{Theorem}

\newtheorem{Defn}[lem]{Definition}
\newtheorem{Ex}[lem]{Example}
\newtheorem{Quest}[lem]{Question}
\newtheorem{Property}[lem]{Property}
\newtheorem{Properties}[lem]{Properties}
\newtheorem{Subprops}{}[lem]
\newtheorem{Para}[lem]{}

\newtheorem{rem}[lem]{Remark}

\newenvironment{exa}{\begin{Ex}\rm}{\end{Ex}}

\newtheorem*{ack*}{ACKNOWLEDGEMENTS}

\newcommand{\pf}{\noindent\begin {proof}}
\newcommand{\epf}{\end{proof}}

\newcommand{\ra}{\rightarrow}
\newcommand{\Ext}{\mbox{\rm Ext}}

\newcommand{\im}{\mbox{\rm im}}

\begin{document}
\begin{center}
{\large  \bf  Totally acyclic complexes and homological invariants over arbitrary rings}

\vspace{0.5cm}  \vspace{0.5cm} Jian Wang$^{a}$, Yunxia Li$^{a}$, Jiangsheng Hu$^{b}$\footnote{Corresponding author.} and Haiyan Zhu$^{c}$\\
$^{a}$College of  Science, Jinling Institute of
Technology, Nanjing 211169, China\\

$^{b}$School of Mathematics, Hangzhou Normal University,
 Hangzhou 311121, China\\

$^{c}$School of Mathematics Sciences, Zhejiang University of Technology, Hangzhou 310023, China\\
E-mail: wangjian@jit.edu.cn, liyunxia@jit.edu.cn, hujs@hznu.edu.cn and hyzhu@zjut.edu.cn \\
\end{center}


\bigskip
\centerline { \bf  Abstract}

\bigskip
\leftskip10truemm \rightskip10truemm \noindent
In this paper, we investigate equivalent characterizations of the condition that every acyclic complex of projective, injective, or flat modules is totally acyclic over a general ring $R$. We provide examples to illustrate relationships among these conditions and show that several are closely tied to the homological invariants silp$(R)$, spli$(R)$ and {\rm sfli}$(R)$. We also give sufficient conditions for the equality spli$(R)$ $=$ silp$(R)$, thereby refining results due to Ballas-Chatzistavridis and Wang-Yang. Further, we extend a result of Christensen-Foxby-Holm on characterizations of Iwanaga-Gorenstein rings to the non-commutative setting. This generalizes a theorem of Estrada-Fu-Iacob, offering additional equivalent characterizations under a general assumption while also yielding characterizations of the Nakayama conjecture.\\
\vbox to 0.3cm{}

\leftskip0truemm \rightskip0truemm
\bigskip
\section { \bf Introduction}
In this paper, all rings $R$ are associative unital rings and all modules are left $R$-modules. We denote by $R^{op}$ the opposite ring of $R$. Recall that an acyclic complex $X^\bullet$ of projective (resp., injective) $R$-modules is \emph{totally acyclic} if  Hom$_R(X^\bullet,M)$ (resp., Hom$_R(M,X^\bullet)$) is an exact sequence for every projective (resp., injective) $R$-module $M$; likewise, an acyclic complex $X^\bullet$ of flat $R$-modules is called
\emph{F-totally acyclic} if $E\otimes_RX^\bullet$ is an exact sequence for each injective $R^{op}$-module $E$. An \emph{Iwanaga-Gorenstein} ring is a left and right Noetherian ring that has finite self-injective dimension on both sides (see \cite{Iwer,Iwer2}). It is well-known that a commutative Gorenstein ring of finite Krull dimension is Iwanaga-Gorenstein. Over an Iwanaga-Gorenstein ring every acyclic complex of projective (resp., injective) modules is totally acyclic and every acyclic complexes of flat modules is $F$-totally acyclic. The study of equivalent characterizations of the condition that every acyclic complex of projective (resp., injective) modules is totally acyclic was initiated in 2006 by Iyengar-Krause \cite{IyHkrAcvstac}, and has been extensively investigated during the last years by Murfet-Salarian \cite{MurSaT}, Estrada-Fu-Iacob \cite{EHuATotal}, Christensen-Kato \cite{LWchrKatoG} and Christensen-Foxby-Holm \cite{LchrFoxHol}. Now, the equivalent characterizations over commutative Noetherian rings are made explicit in the following result.

\begin{thm}{\rm (\cite[Theorem 19.5.25]{LchrFoxHol})}\label{thm:las} Let $R$ be a commutative Noetherian ring. Then the following are equivalent:
\begin{enumerate}
\item[{\rm(1)}] Every acyclic complex of projective $R$-modules is totally acyclic.
\item[{\rm(2)}] Every acyclic complex of injective $R$-modules is totally acyclic.
\item[{\rm(3)}] Every acyclic complex of flat $R$-modules is $F$-totally acyclic.
\item[{\rm(4)}] The local ring $R_{\mathfrak{p}}$ is Iwanaga-Gorenstein for every prime ideal $\mathfrak{p}$ of $R$.
\end{enumerate}
\end{thm}

The main aim of this paper is to study these equivalent conditions for a general ring $R$.
Throughout this paper, we use the following notations.
\begin{enumerate}
\item[{{\rm(1)$^{op}$}}] Every acyclic complex of projective $R^{op}$-modules is totally acyclic.
\item[{{\rm(2)$^{op}$}}] Every acyclic complex of injective $R^{op}$-modules is totally acyclic.
\item[{{\rm(3)$^{op}$}}] Every acyclic complex of flat $R^{op}$-modules is $F$-totally acyclic.
\end{enumerate}

Thanks to Proposition \ref{prop: prop0801} and  Example \ref{em: rem2060501} below, we obtain the following relations among the six conditions above.

\small  $$\xymatrix{&&&&&&&&\\
  (1)\ar[dd]_{{\rm Example \ \ref{em: rem2060501}}(iv),(i)}|\times \ar[rrrr]^{{\rm Example \ \ref{em: rem2060501}}(iv),(v)}|\times&&&&(2)   &&&& (3) \ar[llll]|\times_{{\rm Example \ \ref{em: rem2060501}}(vi),(v)}   \ar[dd]|\times^{{\rm Example \ \ref{em: rem2060501}}(vi),(iii)} \ar@/_3pc/[llllllll]_{\ {\rm Proposition \ \ref{prop: prop0801}}(ii)\Rightarrow(i)}\\
  &&&&&&\\
(1)^{op}&&&&(2)^{op}\ar[rrrr]|\times_{{\rm Example \ \ref{em: rem2060501}}(ii),(iii)}\ar@/^/[uullll]^{\ {\rm Proposition \ \ref{prop: prop0801}}(iii)\Rightarrow(i)}\ar[uu]_{{\rm Example \ \ref{em: rem2060501}}(ii),(v)}|\times\ar[llll]|\times^{{\rm Example \ \ref{em: rem2060501}}(ii),(i)} \ar@/_/[uurrrr]_{\ {\rm Proposition \ \ref{prop: prop0801}}(iii)\Rightarrow(ii)} &&&&(3)^{op} }$$

Note that the conclusion that $(2)^{op}$ implies $(3)$ is essentially taken from \cite[Claim 2.3]{LWchrKatoG} (see Lemma \ref{lem: lemaftac-3}($ii$)). As with the other relations, does $(1)$ imply $(3)$ for arbitrary rings? At this moment, we have no idea how to reach this goal in general,
but there is an interesting case in which this does hold true: every Gorenstein projective $R$-module is Gorenstein flat (see Corollary \ref{cor:2.5}). Recall that an $R$-module $M$ is called \emph{Gorenstein projective} \cite{enochs:gipm} if it is a cycle of a totally acyclic complex of projective $R$-modules, and $M$ is called \emph{Gorenstein flat} \cite{EJT} if it is a cycle of an $F$-totally acyclic complex of flat $R$-modules. It is not at all clear from the definitions that Gorenstein projective $R$-modules are Gorenstein flat. Examples of rings satisfying that every Gorenstein projective $R$-module is Gorenstein flat include but not limited to: right coherent rings with finite left finitistic dimension (see \cite[Proposition 3.4]{H}),  right coherent rings of cardinality at most $\aleph_{n}$ for a fixed natural number $n$ (see \cite[Remarks 6.8 and 2.8]{CCLP}), as well as rings such that every injective $R^{op}$-module has finite flat dimension (see \cite[Remark 2.3($ii$)]{I on the fin}). But for arbitrary rings this is still an open question.

We recall two invariants spli$(R)$ and silp$(R)$ introduced by Gedrich-Gruenberg \cite{Gecomple} and studied by some authors (see, for example \cite{Beli and Reiten,DeTaon,I on certain,I on the fin,I and O on the flat,CUJensen11,Otale1}). More precisely, let spli$(R)$ (resp., silp$(R)$) be the supremum of the projective (resp., injective) lengths of the injective (resp., projective)  $R$-modules.  Whether the equality spli$(R)$ $=$ silp$(R)$ holds is an open question for a general ring $R$
 (see \cite{DemEmaHom,I on certain,Gecomple}). In the special case where
$R$ is an Artin algebra, the equality spli$(R)$ $=$ silp$(R)$ is equivalent to a long-standing conjecture in representation theory, the so-called Gorenstein symmetry conjecture, which is stated as Conjecture 13 at the end of \cite{ARS} (see also \cite{AB CM,Beli and Reiten}).

Recall that a ring $R$ is called \emph{left Gorenstein regular} \cite{EnochMatrix} if the supremum of Gorenstein projective dimensions (\cite{H}) of all  $R$-modules  is finite. It is well-known that $R$ is left Gorenstein regular if and only if  {\rm spli}$(R)$ $=$ {\rm silp}$(R)$ $<$ $\infty$ if and only if {\rm spli}$(R)$ $<$ $\infty$ and  {\rm silp}$(R)$ $<$ $\infty$ (see \cite{Beli and Reiten,I on the fin,EnochMatrix}).

 Let {\rm sfli}$(R)$ be the supremum of the flat lengths of the injective $R$-modules. The invariant sfli$(R)$ has been introduced by Ding-Chen \cite{DC} and studied by some authors (see, for example \cite{Ballas,DLW,I on the fin}). It is clear that sfli$(R)$ $\leq$ spli$(R)$. Following \cite{DB,LWchrisetwG}, we denote by {\rm wGgldim}$(R)$ the supremum of Gorenstein flat dimensions of all $R$-modules.

Our first main result demonstrates that the equivalence among conditions (1)-(3) and their opposites $(1)^{op}$-$(3)^{op}$ above forces the equalities spli$(R)$ $=$ silp$(R)$ and {\rm sfli}$(R)= {\rm sfli}(R^{op})$.

  \begin{thm}\label{ThA} The following statements are true for any ring $R$:
 \begin{enumerate}
\item[{(i)}] If two conditions (1) and (2) are equivalent, then {\rm spli$(R)$ $=$ silp$(R)$}, and therefore $R$ is left Gorenstein regular if and only if  {\rm spli}$(R)$ is finite, or equivalently, {\rm silp}$(R)$ is finite.
\item[{(ii)}] Assume that one of the conditions (a)-(c) is true.
\begin{enumerate}
\item[(a)] Two conditions (2) and (3) are equivalent.
\item [(b)] Two conditions (2) and (2)$^{op}$ are equivalent.
\item[(c)] Two conditions (3) and (3)$^{op}$ are equivalent.
\end{enumerate}
Then {\rm sfli}$(R)= {\rm sfli}(R^{op})$ and {\rm silp}$(R)$ $\leq$ {\rm spli}$(R)$, and therefore {\rm wGgldim}$(R)$ $<$ $\infty$ if and only if  {\rm sfli}$(R)$ $<$ $\infty$.
\end{enumerate}
 \end{thm}
There are two remarks on Theorem \ref{ThA}: $(a)$ If $R$ is a ring with {\rm wGgldim}$(R)$ $<$ $\infty$, then Proposition \ref{lem:lem1501} and \cite[Theorem 5.3]{I on the fin} imply that the six conditions (1)-(3) and $(1)^{op}$-$(3)^{op}$ hold, and thus these conditions are equivalent in this case. However, we employ Example \ref{em: rem2060501-2} to show that there exists a ring $R$ with {\rm wGgldim}$(R)$ $=$ $\infty$ such that the conditions (2) and (2)$^{op}$ are equivalent; $(b)$ Assume that $R$ is a ring which is isomorphic to its opposite ring.   It has been proved by Dalezios and Emmanouil in \cite[Corollary 24]{DemEmaHom} that {\rm silp}$(R)$ $\leq$ {\rm spli}$(R)$. This result can be reobtained by Theorem \ref{ThA} since two conditions $(2)$ and $(2)^{op}$ are equivalent if $R$ is isomorphic to its opposite ring.

A consequence of Theorem \ref{ThA} yields the following result.

\begin{cor}\label{cor:corcothac1} Let $R$ be a ring. Assume that  one of the conditions (a)-(c) in Theorem \ref{ThA}(ii) is true. If the character module of every injective $R^{op}$-module has finite flat dimension, then {\rm spli}$(R)$ $=$ {\rm silp}$(R)$.
\end{cor}
Recall that a ring is called \emph{left coherent} if every finitely generated left ideal is finitely presented. Thanks to Example \ref{em: rem2060501-2}, there exists a ring $R$ satisfying that: (a) The character module of every injective $R^{op}$-module has finite flat dimension; (b) Conditions ${\rm (2)}$ and ${\rm (2)}^{op}$ are equivalent; (c) $R$ is neither isomorphic to $R^{op}$ nor right coherent. Consequently, Corollary \ref{cor:corcothac1} improves \cite[Corollary 3.4]{B-O-cohe}, where the equality {\rm spli}$(R)$ $=$ {\rm silp}$(R)$ holds under the assumptions that $R$ is a left and right coherent ring and $R$ is isomorphic to the opposite ring $R^{op}$ (see Corollary \ref{corollary:2.8}(1)).

On the other hand, recall that a ring $R$ is  called \emph{left  generalized coherent} if all level  $R^{op}$-modules have finite flat dimension (see \cite{PJwandGYang}), where
an $R$-module $N$ is called \emph{level} if $\textrm{Tor}_{1}^{R}(M,N)=0$ for any $R^{op}$-module $M$ with a projective resolution by finitely generated projectives (see \cite{BGGIAC}). According to \cite[Theorem 2.12]{BGGIAC}, the character module of every injective $R$-module is a level $R^{op}$-module. Then, for any commutative generalized coherent ring $R$, conditions ${\rm (2)}$ and ${\rm (2)}^{op}$ are equivalent and the character module of every injective $R$-module has finite flat dimension. Hence Corollary \ref{cor:corcothac1} improves \cite[Corollary 3.16]{PJwandGYang}, where the equality {\rm spli}$(R)$ $=$ {\rm silp}$(R)$ holds under the assumption that $R$ is a commutative generalized coherent ring (see Corollary \ref{corollary:2.8}(2)).

Our final theorem is a non-commutative version of Theorem \ref{thm:las}.

\begin{thm} \label{cor: thmeq2-1} Let $R$ be a left and right Noetherian ring and $0\ra R \ra E^0(R)\ra \cdots \ra E^{i}(R)\ra \cdots$  the minimal injective coresolution of  the $R$-module $R$.  Assume that there is a nonnegative integer $n$ such that the flat dimension of $E^i(R)$ is at most $n$  for all $i$ $\geq$ $0$. Then the following are equivalent:
\begin{enumerate}
\item[(i)] $R$ is Iwanaga-Gorenstein.
\item[(ii)] All acyclic complexes of injective $R$-modules or $R^{op}$-modules are totally acyclic.

\item[(iii)] All acyclic complexes of projective $R$-modules or $R^{op}$-modules are totally acyclic and all Gorenstein projective $R$-modules or $R^{op}$-modules are Gorenstein flat.

\item[(iv)] All acyclic complexes of projective $R$-modules or $R^{op}$-modules are $F$-totally acyclic.

\item[(v)] All acyclic complexes of flat $R$-modules or $R^{op}$-modules are $F$-totally acyclic.

\item[(vi)]  All acyclic complexes of injective $R$-modules are totally acyclic and all acyclic complexes of flat $R$-modules are $F$-totally acyclic.
\end{enumerate}
\end{thm}
We observe that conditions $(iv)$ and $(v)$ in Theorem \ref{cor: thmeq2-1} are equivalent for any ring $R$. This fact was established by Christensen-Kato in \cite[Claim 2.2]{LWchrKatoG} (see also Lemma \ref{lem: lemaftac-3}$(i)$). If $R$ is a left and right Noetherian ring of finite finitistic flat dimension that satisfies the Auslander condition, then the equivalence of conditions $(i)$ and $(ii)$ in Theorem \ref{cor: thmeq2-1} has been proved by Estrada-Fu-Iacob in \cite[Theorem 7]{EHuATotal}. Recall that  $R$ is said to satisfy the \emph{Auslander condition} if the flat dimension of the $i$th term in a minimal injective coresolution of the $R$-module $R$ is at most $i-1$ for all $i\geq1$. Furthermore, $R$ has \emph{finite finitistic flat dimension} if the maximum of flat dimensions among the modules with finite flat dimension is finite.

Let $R$ be the path algebra of the quiver $\xymatrix{\bullet &\ar[l]{\bullet}\ar[r]&{\bullet}}$ over a field. Then $R$ is a left and right Noetherian ring such that it is also hereditary. Consequently, $R$ is an Iwanaga-Gorenstein ring and therefore satisfies the hypothesis of Theorem \ref{cor: thmeq2-1} by \cite[Theorem 9.1.10]{EJ}. However, $R$ does not satisfy the Auslander condition (see \cite[p.12]{ARk-G}). This implies that $R$ does not satisfy the hypothesis of Theorem 7 in \cite{EHuATotal}. Thus, our Theorem \ref{cor: thmeq2-1} improves \cite[Theorem 7]{EHuATotal} by adding more equivalent characterizations under a general assumption.

Recall that a finite-dimensional algebra $R$ has \emph{infinite dominant dimension} if in a minimal injective coresolution $0\ra R \ra E^0(R)\ra \cdots \ra E^{i}(R)\ra \cdots$ of the $R$-module $R$, all the $E^{i}(R)$ are projective. In the representation theory of algebras, the long-standing and not yet solved Nakayama conjecture says that a finite-dimensional algebra over a field with infinite dominant dimension is self-injective \cite{Nakayama} (see also \cite[Conjecture (8), p. 41]{ARS}).

Applying Theorem \ref{cor: thmeq2-1} to finite-dimensional algebras, our characterizations of the Nakayama conjecture read as follows.

\begin{cor} {\rm(Corollary \ref{cor:2.10})}\label{cor:1.5} Let $R$ be a finite-dimensional algebra over a field with infinite dominant dimension. Then the following are equivalent:
\begin{enumerate}
\item[(i)] $R$ is self-injective.
\item[(ii)] All acyclic complexes of injective $R$-modules or $R^{op}$-modules are totally acyclic.

\item[(iii)] All acyclic complexes of projective $R$-modules or $R^{op}$-modules are totally acyclic.
\end{enumerate}
\end{cor}

The proofs of the above results will be carried out in the next section.

\section{\bf Proofs of the results}
Throughout the paper, $R$ denotes an associative unital ring. By an $R$-module we
mean a left $R$-module; right $R$-modules are modules over the opposite ring $R^{op}$. Let $M$ be an $R$-module. We use pd$_R M$ (resp., id$_R M$ and fd$_R M$) to denote the projective (resp., injective and flat) dimension of $M$.
Chain complexes of $R$-modules,
 are indexed homologically. That is, the degree $n$ component of
an $R$-complex $X^\bullet$ is denoted $X_n$, and we denote the degree $n$ cycle by $Z_{n}(X^\bullet)$. For an $R$-module $M$ and a complex $X^\bullet$ of $R$-modules, $M^+$ is the character module Hom$_\mathbb{Z}(M,\mathbb{Q}/\mathbb{Z})$  and ${X^\bullet}^+$ is the complex Hom$_\mathbb{Z}(X^\bullet, \mathbb{Q}/\mathbb{Z})$. If there is an acyclic complex  $$\xymatrix{\cdots\ar[r]&P_{i}\ar[r]^{f_{i}}&\cdots\ar[r]&P_{0}\ar[r]^{f_{0}}&M\ar[r]&0}$$ of $R$-modules with each $P_i$ projective, then $\im{(f_{n})}$ is called an \emph{$n$-th syzygy} of $M$. We denote it by $\Omega^{n}(M)$ $($$\Omega^{0}(M)$ $=$ $M$$)$. Dually, one can define an \emph{$n$-th cosyzygy}  $\Omega^{-n}(M)$ of $M$. For the unexplained notations, the reader is referred to \cite{LchrFoxHol,EJ}.

Next, we fix the following notations throughout the paper.

$\mathbf{K}_{\rm ac}({\rm Prj}R)$ (resp., $\mathbf{K}_{\rm ac}({\rm Prj}R^{op})$) is the  homotopy category of the acyclic complexes of projective  $R$-modules (resp., $R^{op}$-modules);

$\mathbf{K}_{\rm ac}({\rm Inj}R)$ (resp., $\mathbf{K}_{\rm ac}({\rm Inj}R^{op})$) is the homotopy category of the acyclic complexes of injective  $R$-modules (resp., $R^{op}$-modules);

$\mathbf{K}_{\rm tac}({\rm Prj}R)$ (resp., $\mathbf{K}_{\rm tac}({\rm Prj}R^{op})$) is the  homotopy category of the totally acyclic complexes of projective  $R$-modules (resp., $R^{op}$-modules);

$\mathbf{K}_{\rm tac}({\rm Inj}R)$ (resp., $\mathbf{K}_{\rm tac}({\rm Inj}R^{op})$) is the homotopy category of the totally acyclic complexes of injective  $R$-modules (resp., $R^{op}$-modules);

$\mathbf{K}_{\rm ac}({\rm Flat}R)$ (resp., $\mathbf{K}_{\rm ac}({\rm Flat}R^{op})$) is the  homotopy category of the acyclic complexes of flat  $R$-modules (resp., $R^{op}$-modules);

$\mathbf{K}_{\rm tac}({\rm Flat}R)$ (resp., $\mathbf{K}_{\rm tac}({\rm Flat}R^{op})$) is the  homotopy category of the F-totally acyclic complexes of flat  $R$-modules (resp., $R^{op}$-modules).

For simplicity, we can rewrite the conditions (1)-(3) and $(1)^{op}$-$(3)^{op}$ in the introduction as follows.

\begin{enumerate}

\item[{{($i$)}}] Condition {\rm(1)}: $\mathbf{K}_{\rm ac}({\rm Prj}R)$ $=$ $\mathbf{K}_{\rm tac}({\rm Prj}R)$.
\item[{{($ii$)}}] Condition {\rm(2)}: $\mathbf{K}_{\rm ac}({\rm Inj}R)$ $=$ $\mathbf{K}_{\rm tac}({\rm Inj}R)$.
\item[{{($iii$)}}] Condition {\rm(3)}: $\mathbf{K}_{\rm ac}({\rm Flat}R)$ $=$ $\mathbf{K}_{\rm tac}({\rm Flat}R)$.
\item[($iv$)] Condition {\rm(1)$^{op}$}: $\mathbf{K}_{\rm ac}({\rm Prj}R^{op})$ $=$ $\mathbf{K}_{\rm tac}({\rm Prj}R^{op})$.
\item[($v$)] Condition {\rm(2)$^{op}$}: $\mathbf{K}_{\rm ac}({\rm Inj}R^{op})$ $=$ $\mathbf{K}_{\rm tac}({\rm Inj}R^{op})$.
\item[($vi$)] Condition {\rm(3)$^{op}$}: $\mathbf{K}_{\rm ac}({\rm Flat}R^{op})$ $=$ $\mathbf{K}_{\rm tac}({\rm Flat}R^{op})$.
\end{enumerate}

The following lemma is essentially taken from the proof of \cite[Claims 2.2 and 2.3]{LWchrKatoG}, where a variation of it appears. The proof given there carries over to the present situation.

\begin{lem} \label{lem: lemaftac-3} The following conditions are true for any ring $R$:
\begin{enumerate}

\item[{(i)}] {\rm(\cite[Claim 2.2]{LWchrKatoG} )} $\mathbf{K}_{\rm ac}({\rm Prj}R)$  $\subseteq$ $\mathbf{K}_{\rm tac}({\rm Flat}R)$ if and only if $\mathbf{K}_{\rm ac}({\rm Flat}R)$ $=$ $\mathbf{K}_{\rm tac}({\rm Flat}R)$.
   \item[{(ii)}] {\rm(\cite[Claim 2.3]{LWchrKatoG})}  If $\mathbf{K}_{\rm ac}({\rm Inj}R^{op})$ $=$ $\mathbf{K}_{\rm tac}({\rm Inj}R^{op})$, then $\mathbf{K}_{\rm ac}({\rm Flat}R)$ $=$ $\mathbf{K}_{\rm tac}({\rm Flat}R)$.

\end{enumerate}
\end{lem}

The following lemma has been proved by $\check{\rm{S}}$t'ov\'{\i}$\check{\rm{c}}$ek in \cite[Corollary 5.9, p.23]{Onpurityand}. One can obtain  a different proof by using Theorem 5.1 in \cite{Periodic} and dimension shifting.
\begin{lem}\label{lem:lem080205101} Let $R$ be a ring. If $M$ is a cycle of an acyclic complex of injective  $R$-modules, then $\Ext^1_R(F,M)$ $=$ $0$ for each $R$-module $F$ with finite flat dimension.
\end{lem}

Using \cite[Theorem 5.6]{JJarxiv} and the proof of \cite[Lemma 1]{IcaEnocCov}, one can obtain the following result.

\begin{lem}\label{lem:lemongif-1} Let $R$ be a left Noetherain ring. Assume that the character module of every Gorenstein
injective $R$-module is a Gorenstein flat $R^{op}$-module. If $M$ is an $R$-module such that the character module $M^+$ is a Gorenstein flat $R^{op}$-module, then $M$ is Gorenstein injective.
\end{lem}

Following \cite{JJarxiv}, an $R$-module is called \emph{projectively coresolved Gorenstein flat} if it is a cycle of an $F$-totally acyclic complex of projective $R$-modules.

\begin{prop} \label{prop: prop0801} Let $R$ be a  ring.  Consider the following conditions:
\begin{enumerate}
\item[{(i)}] $\mathbf{K}_{\rm ac}({\rm Prj}R)$ $=$ $\mathbf{K}_{\rm tac}({\rm Prj}R)$ $\subseteq$ $\mathbf{K}_{\rm tac}({\rm Flat}R)$.
\item[{(ii)}] $\mathbf{K}_{\rm ac}({\rm Flat}R)$ $=$  $\mathbf{K}_{\rm tac}({\rm Flat}R)$.
\item[{(iii)}] $\mathbf{K}_{\rm ac}({\rm Inj}R^{op})$ $=$  $\mathbf{K}_{\rm tac}({\rm Inj}R^{op})$.
\end{enumerate}

Then $(i)$ $\Leftrightarrow$ $(ii)$ $\Leftarrow$ $(iii)$. Moreover, if $R$ is right Noetherian, then $(i)$ $\Leftrightarrow$ $(ii)$ $\Leftrightarrow$ $(iii)$.
\end{prop}
\begin{proof} $(i)$ $\Rightarrow$ $(ii)$. Note that $\mathbf{K}_{\rm ac}({\rm Prj}R)$ $\subseteq$ $\mathbf{K}_{\rm tac}({\rm Flat}R)$ by $(i)$. It follows from Lemma \ref{lem: lemaftac-3}$(i)$ that $\mathbf{K}_{\rm ac}({\rm Flat}R)$ $=$  $\mathbf{K}_{\rm tac}({\rm Flat}R)$.

$(ii)$ $\Rightarrow$ $(i)$. If ${P^\bullet}$ is a complex  in $\mathbf{K}_{\rm ac}({\rm Prj}R)$, then ${P^\bullet}$ is in $\mathbf{K}_{\rm ac}({\rm Flat}R)$, and therefore ${P^\bullet}$ is in $\mathbf{K}_{\rm tac}({\rm Flat}R)$ by $(ii)$. Hence each cycle of ${P^\bullet}$ is a projectively coresolved Gorenstein flat $R$-module. By \cite[Theorem 4.4]{JJarxiv}, for any $n$-th cycle $M_n$ of the complex ${P^\bullet}$, Ext$_R^1(M_n,P)$ $=$ $0$ for all projective $R$-modules $P$ and all $n$ $\in$ $\mathbb{Z}$. Then  ${P^\bullet}$ is in $\mathbf{K}_{\rm tac}({\rm Prj}R)$, and so $\mathbf{K}_{\rm ac}({\rm Prj}R)$ $=$ $\mathbf{K}_{\rm tac}({\rm Prj}R)$. It is clear that $\mathbf{K}_{\rm tac}({\rm Prj}R)$ $\subseteq$ $\mathbf{K}_{\rm ac}({\rm Flat}R)$. So $\mathbf{K}_{\rm tac}({\rm Prj}R)$ $\subseteq$ $\mathbf{K}_{\rm tac}({\rm Flat}R)$ by $(ii)$, as desired.

$(iii)$ $\Rightarrow$ $(ii)$ follows from Lemma \ref{lem: lemaftac-3}$(ii)$.

$(ii)$ $\Rightarrow$ $(iii)$. Let ${E^\bullet}$ be a complex in $\mathbf{K}_{\rm ac}({\rm Inj}R^{op})$ and $K_n$  the $n$-th cycle of the complex ${E^\bullet}$.  Then ${K_n}^+$ is a cycle of the complex ${E^\bullet}^+$  for all $n$ $\in$ $\mathbb{Z}$.  Note that $R$ is right Noetherian by hypothesis. Then ${E^\bullet}^+$ is in $\mathbf{K}_{\rm ac}({\rm Flat}R)$, and so  ${K_n}^+$ is Gorenstein flat for all $n$ $\in$ $\mathbb{Z}$ by $(ii)$. It follows that the character module of every Gorenstein
injective $R^{op}$-module is a Gorenstein flat $R$-module. By Lemma \ref{lem:lemongif-1}, $K_n$ is Gorenstein injective for all  $n$ $\in$ $\mathbb{Z}$. Thus ${E^\bullet}$ is in $\mathbf{K}_{\rm tac}({\rm Inj}R^{op})$. This completes the proof.
\end{proof}
\begin{rem}\label{rem:1} We note that the equivalence $(i) \Leftrightarrow (ii)$ and the implication $(iii) \Rightarrow (ii)$ in Proposition \ref{prop: prop0801} were already observed in \cite[Remark 2.5]{I and O Total}. For the convenience of the reader, we present a slightly different proof here.
\end{rem}

Without extra assumptions on the ring, we do not know if the equivalence of two conditions (1) and (3) in Theorem \ref{thm:las} is true for a general ring. In general, the relation between these conditions is tied to an unresolved problem in Gorenstein homological algebra. Corollary \ref{cor:2.5} explains how.

\begin{cor}\label{cor:2.5} Let $R$ be a ring. If every Gorenstein projective $R$-module is Gorenstein flat, then
$\mathbf{K}_{\rm ac}({\rm Proj}R)$ $=$ $\mathbf{K}_{\rm tac}({\rm Proj}R)$ if and only if $\mathbf{K}_{\rm ac}({\rm Flat}R)$ $=$ $\mathbf{K}_{\rm tac}({\rm Flat}R)$.
\end{cor}
\begin{proof} Note that $\mathbf{K}_{\rm tac}({\rm Prj}R)$ $\subseteq$ $\mathbf{K}_{\rm tac}({\rm Flat}R)$ by \cite[Theorem 2.2]{I on the fin}. So the result follows from Proposition \ref{prop: prop0801}.
\end{proof}

\begin{prop}\label{lem:lem1501}  The following conditions are true for any ring $R$:
\begin{enumerate}
\item[{(i)}]  If {\rm silp}$(R)$ $<$ $\infty$, then $\mathbf{K}_{\rm ac}({\rm Prj}R)$ $=$  $\mathbf{K}_{\rm tac}({\rm Prj}R)$.
\item[{(ii)}] If {\rm sfli}$(R)$ $<$ $\infty$, then $\mathbf{K}_{\rm ac}({\rm Inj}R)$ $=$  $\mathbf{K}_{\rm tac}({\rm Inj}R)$ and $\mathbf{K}_{\rm ac}({\rm Flat}R^{op})$ $=$  $\mathbf{K}_{\rm tac}({\rm Flat}R^{op})$.
\end{enumerate}
\end{prop}
\begin{proof} $(i)$ It suffices to show $\mathbf{K}_{\rm ac}({\rm Prj}R)$ $\subseteq$  $\mathbf{K}_{\rm tac}({\rm Prj}R)$. Let $M$ be any projective $R$-module and ${P^\bullet}$ a complex  in $\mathbf{K}_{\rm ac}({\rm Prj}R)$. Note that {\rm silp}$(R)$ $<$ $\infty$ by assumption. Then $M$ has finite injective dimension. For each cycle $N$ of ${P^\bullet}$,  one can check that Ext$^1_R(N,M)$ $=$ $0$ by dimension shifting. Then the sequence Hom$_R({P^\bullet},M)$ is exact. It follows that ${P^\bullet}$ is in $\mathbf{K}_{\rm tac}({\rm Prj}R)$, as desired.

$(ii)$  Assume that ${E^\bullet}$ is a complex  in $\mathbf{K}_{\rm ac}({\rm Inj}R)$ and $M$ is a cycle of ${E^\bullet}$. Since {\rm sfli}$(R)$ $<$ $\infty$ by assumption, $\Ext^1_R(E,M)$ $=$ $0$ for every injective $R$-module $E$ by Lemma \ref{lem:lem080205101}. It follows that Hom$_R(E,{E^\bullet})$ is an exact sequence for each injective $R$-module $E$.  Thus ${E^\bullet}$ is in $\mathbf{K}_{\rm tac}({\rm Inj}R)$, and hence $\mathbf{K}_{\rm ac}({\rm Inj}R)$ $=$  $\mathbf{K}_{\rm tac}({\rm Inj}R)$. So Proposition \ref{prop: prop0801} yields $\mathbf{K}_{\rm ac}({\rm Flat}R^{op})$ $=$ $\mathbf{K}_{\rm tac}({\rm Flat}R^{op})$.
\end{proof}

\begin{lem} \label{lem: lemactac120701} The following conditions are true for any ring $R$:
\begin{enumerate}
 \item[{(i)}] If {\rm spli}$(R)$ $<$ $\infty$, then, for each $R$-module $M$, $\Omega^n(M)$ is a cycle of  a complex  in $\mathbf{K}_{\rm ac}({\rm Prj}R)$, where $n$ $=$ {\rm spli}$(R)$.
 \item[{(ii)}] If {\rm sfli}$(R)$ $<$ $\infty$, then, for each $R$-module $M$, $\Omega^n(M)$ is a cycle of  a complex  in $\mathbf{K}_{\rm ac}({\rm Flat}R)$, where $n$ $=$ {\rm sfli}$(R)$.
\item[{(iii)}] If {\rm silp}$(R)$ $<$ $\infty$, then, for each $R$-module $M$, $\Omega^{-n}(M)$ is a cycle of  a complex  in $\mathbf{K}_{\rm ac}({\rm Inj}R)$, where $n$ $=$ {\rm silp}$(R)$.

\item[{(iv)}] If all terms of an acyclic complex ${Q^\bullet}$ have finite projective (resp., flat) dimensions  at most  $n$, then, for each cycle $M$ of the complex ${Q^\bullet}$, $\Omega^n(M)$ is a cycle of a complex  in $\mathbf{K}_{\rm ac}({\rm Prj}R)$ (resp., $\mathbf{K}_{\rm ac}({\rm Flat}R)$).

\end{enumerate}
\end{lem}
\begin{proof} Assertions $(i)$ and $(iii)$ are established in \cite[4.1]{Gecomple}, and the proof given there applies equally to assertions $(ii)$ and $(iv)$.
\end{proof}

\begin{lem} \label{lem: lemactac120701-11} Let $R$ be a ring. If the character module of every injective  $R^{op}$-module has finite flat dimension, then {\rm sfli}$(R^{op})$ $\leq$ {\rm silp}$(R)$.
\end{lem}
\begin{proof}  If silp$(R)$ $=$ $\infty$, then {\rm sfli}$(R^{op})$ $\leq$ {\rm silp}$(R)$. If silp$(R)$ $<$ $\infty$, we set {\rm silp}$(R)$ $=$ $n$. Let $M$ be any injective  $R^{op}$-module. Then  {\rm fd}$_R M^+$ $<$ $\infty$ by hypothesis. Note that the supremum of the projective dimensions of those $R$-modules with finite projective dimension is at most $n$ by \cite[Proposition 1.3($i$)]{I on the fin}. It follows from \cite[Proposition 6]{CUJensen} that every flat $R$-module has finite projective dimension at most $n$, and therefore {\rm pd}$_R M^+$ $<$ $\infty$. Applying \cite[Proposition 1.3($i$)]{I on the fin} again, we have {\rm pd}$_R M^+$ $\leq$ $n$. Thus one can check that  {\rm id}$_R M^+$ $\leq$ $n$, and hence {\rm fd}$_R M$ $\leq$ $n$. This shows sfli$(R^{op})$ $\leq$ {\rm silp}$(R)$, as needed.
\end{proof}

{\bf Proof of Theorem \ref{ThA}.}
$(i)$ If spli$(R)$ $=$ $\infty$ and silp$(R)$ $=$ $\infty$, then {\rm silp}$(R)$ $=$ {\rm spli}$(R)$.
If spli$(R)$ $<$ $\infty$,  we assume that spli$(R)$ $=$ $n$. Then, for any $R$-module $M$,  $\Omega^{n}(M)$ is a cycle of a complex  in $\mathbf{K}_{\rm ac}({\rm Prj}R)$ by Lemma \ref{lem: lemactac120701}$(i)$.
Since  spli$(R)$ $<$ $\infty$, one has $\mathbf{K}_{\rm ac}({\rm Inj}R)$ $=$  $\mathbf{K}_{\rm tac}({\rm Inj}R)$, and the proof is similar to that of Proposition \ref{lem:lem1501}$(i)$. Consequently, we obtain $\mathbf{K}_{\rm ac}({\rm Prj}R)$ $=$  $\mathbf{K}_{\rm tac}({\rm Prj}R)$ by assumption. This implies that $\Omega^{n}(M)$ is a cycle of  a complex in $\mathbf{K}_{\rm tac}({\rm Prj}R)$. Let $Q$ be any projective $R$-module. Thus Ext$^1_R(\Omega^{n}(M),Q)$ $=$ $0$ for every  $R$-module $M$. It follows that Ext$^{n+1}_R(M,Q)$ $=$ $0$ for every $R$-module $M$. This shows id$_R Q$ $\leq$ $n$, and therefore silp$(R)$ $<$ $\infty$. So we have silp$(R)$ $=$ spli$(R)$ by \cite[Proposition 1.3($iv$)]{I on the fin}.

If silp$(R)$ $<$ $\infty$, we assume that silp$(R)$ $=$ $n$. Then, for each  $R$-module $N$,  $\Omega^{-n}(N)$ is a cycle of a complex  in $\mathbf{K}_{\rm ac}({\rm Inj}R)$ by Lemma \ref{lem: lemactac120701}$(iii)$.  Note that $\mathbf{K}_{\rm ac}({\rm Inj}R)$ $=$  $\mathbf{K}_{\rm tac}({\rm Inj}R)$ by assumption. Then $\Omega^{-n}(N)$ is a cycle of a complex  in $\mathbf{K}_{\rm tac}({\rm Inj}R)$.  Let $E$ be any injective  $R$-module.  Then Ext$^1_R(E,\Omega^{-n}(N))$ $=$ $0$ for every $R$-module $N$. It follows that Ext$^{n+1}_R(E,N)$ $=$ $0$ for every  $R$-module $N$. Thus pd$_RE$ $\leq$ $n$. So spli$(R)$ $<$ $\infty$. By \cite[Proposition 1.3($iv$)]{I on the fin}, spli$(R)$ $=$ silp$(R)$. This completes the proof.

By the proof above, spli$(R)$ $=$ silp$(R)$. Applying \cite[Proposition 2.2]{EnochMatrix}, we obtain the desired characterizations of the left Gorenstein regular ring $R$.

$(ii)$ To prove {\rm sfli}$(R)$ $=$ {\rm sfli}$(R^{op})$, we only need to check that the equality holds provided one of  the two invariants is finite. Similarly, to show {\rm silp}$(R)$ $\leq$ {\rm spli}$(R)$, it suffices to prove that the inequality holds provided that {\rm spli}$(R)$ is finite.

$(a)$ Assume that the two  conditions ${\rm (2)}$ and ${\rm (3)}$ are equivalent.

If sfli$(R)$ $<$ $\infty$, then $\mathbf{K}_{\rm ac}({\rm Inj}R)$ $=$ $\mathbf{K}_{\rm tac}({\rm Inj}R)$ by Proposition \ref{lem:lem1501}$(ii)$. Hence $\mathbf{K}_{\rm ac}({\rm Flat}R)$ $=$ $\mathbf{K}_{\rm tac}({\rm Flat}R)$ by assumption. Let {\rm sfli}$(R)$ $=$ $n$. For any  $R$-module $N$, $\Omega^{n}(N)$ is a cycle of a complex ${F^\bullet}$ in $\mathbf{K}_{\rm ac}({\rm Flat}R)$ by Lemma \ref{lem: lemactac120701}$(ii)$. It follows that $E\otimes_R{F^\bullet}$ is acyclic for every injective $R^{op}$-module $E$. Thus Tor$_1^R(E,\Omega^n (N))$ $=$ $0$ for every injective  $R^{op}$-module $E$. Then Tor$_{n+1}^R(E, N)$ $=$ $0$. It follows that fd$_{R^{op}}E$ $\leq$ $n$. So sfli$(R^{op})$ $\leq$ $n$. By \cite[Theorem 5.3]{I on the fin}, {\rm sfli}$(R)$ $=$ {\rm sfli}$(R^{op})$.

If sfli$(R^{op})$ $<$ $\infty$, then $\mathbf{K}_{\rm ac}({\rm Inj}R^{op})$ $=$  $\mathbf{K}_{\rm tac}({\rm Inj}R^{op})$ by Proposition \ref{lem:lem1501}$(ii)$, and so we obtain $\mathbf{K}_{\rm ac}({\rm Flat}R)$ $=$  $\mathbf{K}_{\rm tac}({\rm Flat}R)$ by Proposition \ref{prop: prop0801}. Then $\mathbf{K}_{\rm ac}({\rm Inj}R)$ $=$  $\mathbf{K}_{\rm tac}({\rm Inj}R)$ by assumption. It follows from Proposition \ref{prop: prop0801} that $\mathbf{K}_{\rm ac}({\rm Flat}R^{op})$ $=$  $\mathbf{K}_{\rm tac}({\rm Flat}R^{op})$, and therefore we have  sfli$(R)$ $=$ sfli$(R^{op})$ by the proof in the paragraph above.

By the proof above  sfli$(R)$ $=$ sfli$(R^{op})$. Applying \cite[Theorem 5.3]{I on the fin},  we obtain the characterization of the finiteness of wGgldim$(R)$.

Next we prove that {\rm silp}$(R)$ $\leq$ {\rm spli}$(R)$.   If spli$(R)$ $<$ $\infty$, then sfli$(R)$ $<$ $\infty$ by noting that  sfli$(R)$ $\leq$ spli$(R)$. By the proof above, we have {\rm sfli}$(R)$ $=$ {\rm sfli}$(R^{op})$, and so  sfli$(R^{op})$ $<$ $\infty$. So \cite[Corollary 22]{DemEmaHom} shows {\rm silp}$(R)$ $=$ {\rm spli}$(R)$.

$(b)$ Assume that the two  conditions ${\rm (2)}$ and ${\rm (2)}^{op}$ are equivalent.

If sfli$(R)$ $<$ $\infty$, then the condition (2)  holds by Proposition \ref{lem:lem1501}$(ii)$, and so the condition (2)$^{op}$ holds by assumption. Applying  Proposition \ref{prop: prop0801}, the condition (3) holds. By $(ii)(a)$, we have {\rm sfli}$(R)$ $=$ {\rm sfli}$(R^{op})$.

 If sfli$(R^{op})$ $<$ $\infty$, then we have {\rm sfli}$((R^{op})^{op})$ $=$ {\rm sfli}$(R^{op})$. So {\rm sfli}$(R)$ $=$ {\rm sfli}$(R^{op})$, as needed.

If spli$(R)$ $<$ $\infty$, then  the condition (2) holds. This implies that the condition (2)$^{op}$ holds by assumption. Applying Proposition \ref{prop: prop0801}, the condition (1) holds. By $(i)$, we have {\rm spli}$(R)$ $=$ {\rm silp}$(R)$, and therefore silp$(R)$ $\leq$ spli$(R)$.

$(c)$ Assume that the two  conditions ${\rm (3)}$ and ${\rm (3)}^{op}$ are equivalent.

If sfli$(R)$ $<$ $\infty$, then the condition (2)  holds by Proposition \ref{lem:lem1501}$(ii)$, and so the condition (3)$^{op}$ holds by Proposition \ref{prop: prop0801}. By assumption, the condition (3) holds. By $(ii)(a)$, we have {\rm sfli}$(R)$ $=$ {\rm sfli}$(R^{op})$.

If sfli$(R^{op})$ $<$ $\infty$, then we have {\rm sfli}$((R^{op})^{op})$ $=$ {\rm sfli}$(R^{op})$ by the proof above. Hence {\rm sfli}$(R)$ $=$ {\rm sfli}$(R^{op})$.

If spli$(R)$ $<$ $\infty$, then  sfli$(R)$ $<$ $\infty$. Hence the two conditions (2) and ${\rm (3)}^{op}$ hold by Proposition \ref{lem:lem1501}$(ii)$, and so the condition (3) holds by assumption.  By $(ii)(a)$, we have {\rm silp}$(R)$ $\leq$ {\rm spli}$(R)$.
This completes the proof. \hfill$\Box$

{\bf Proof of  Corollary \ref{cor:corcothac1}.}
Note that {\rm silp}$(R)$ $\leq$ {\rm spli}$(R)$ by Theorem \ref{ThA}$(ii)$. If {\rm silp}$(R)$ $\neq$ {\rm spli}$(R)$, then silp$(R)$ $<$ $\infty$. By Lemma \ref{lem: lemactac120701-11}, we have sfli$(R^{op})$ $\leq$ silp$(R)$ $<$ $\infty$. Then sfli$(R)$ $=$ sfli$(R^{op})$ $<$ $\infty$ by Theorem \ref{ThA}$(ii)$, and so the $n$-th cozyzygy of any $R$-module is Gorenstein injective. It follows from \cite[Theorem 4.1]{I on the fin} that {\rm silp}$(R)$ $=$ {\rm spli}$(R)$. This leads to a contradiction. Hence {\rm silp}$(R)$ $=$ {\rm spli}$(R)$, as needed.\hfill$\Box$

Recall from the introduction that a ring $R$ is called \emph{left generalized coherent} if all level $R^{op}$-modules
have finite flat dimension. We refer to \cite{PJwandGYang} for a detailed discussion on this matter. The following corollary is a direct consequence of Corollary \ref{cor:corcothac1}.
\begin{cor}\label{corollary:2.8} Let $R$ be a ring.
\begin{enumerate}
\item[(i)] {\rm(}\cite[Corollary 3.4]{B-O-cohe}{\rm)} If $R$ is a left and right coherent ring such that $R$ is isomorphic to the opposite ring $R^{op}$, then {\rm spli}$(R)$ $=$ {\rm silp}$(R)$.

\item[(ii)] {\rm(}\cite[Corollary 3.16]{PJwandGYang}{\rm)} If $R$ is a commutative generalized coherent ring, then {\rm spli}$(R)$ $=$ {\rm silp}$(R)$.
\end{enumerate}
\end{cor}
\begin{proof} $(i)$ By \cite[Theorem 1]{ChSflat},  the character module of every injective $R^{op}$-module is a flat $R$-module provided that $R$ is right coherent. Now, $(i)$ is a direct consequence of
Corollary \ref{cor:corcothac1}.

$(ii)$  According to \cite[Theorem 2.12]{BGGIAC}, the character module  of every injective $R$-module is a level $R^{op}$-module. Note that $R$ is a commutative generalized coherent ring. It follows that the character module of every injective $R$-module has finite flat dimension. So $(ii)$ follows from Corollary \ref{cor:corcothac1}.
\end{proof}

{\bf Proof of  Theorem \ref{cor: thmeq2-1}.}
$(i)$ $\Rightarrow$ $(ii)$. If $R$ is Iwanaga-Gorenstein, then {\rm sfli$(R)$} $<$ $\infty$ and {\rm sfli$(R^{op})$} $<$ $\infty$ by \cite[Theorem 9.1.11]{EJ}. So Proposition \ref{lem:lem1501}$(ii)$ yields the desired result.

$(ii)$ $\Rightarrow$ $(iii)$. By Proposition \ref{prop: prop0801}, all acyclic complexes of projective $R$-modules or $R^{op}$-modules are totally acyclic. Let $M$ be any Gorenstein projective $R$-module. Then $M$ is a cycle of  a complex ${P^\bullet}$ in $\mathbf{K}_{\rm tac}({\rm Prj}R)$, and therefore $M^+$ is a cycle of the complex  ${P^\bullet}^+$ in $\mathbf{K}_{\rm ac}({\rm Inj}R^{op})$. Note that $\mathbf{K}_{\rm ac}({\rm Inj}R^{op})$ $=$  $\mathbf{K}_{\rm tac}({\rm Inj}R^{op})$ by $(ii)$. It follows that $M^+$ is Gorenstein injective. By \cite[Theorem 2.2]{I on the fin}, $M$ is Gorenstein flat.  Hence we obtain that all Gorenstein projective  $R$-modules are Gorenstein flat. Similarly, one can prove the result over the ring $R^{op}$.

$(iii)$ $\Rightarrow$ $(iv)$.  Let ${P^\bullet}$  be a complex in $\mathbf{K}_{\rm ac}({\rm Prj}R)$. Then the complex ${P^\bullet}$ is in $\mathbf{K}_{\rm ac}({\rm Flat}R)$. By  $(iii)$ and Corollary \ref{cor:2.5}, we obtain that ${P^\bullet}$ is in $\mathbf{K}_{\rm tac}({\rm Flat}R)$. Similarly, one can check that all acyclic complexes of projective $R^{op}$-modules are $F$-totally acyclic.

 $(iv)$ $\Rightarrow$ $(v)$ follows from Lemma \ref{lem: lemaftac-3}$(i)$.

$(v)$ $\Rightarrow$ $(vi)$  follows from Proposition \ref{prop: prop0801}.

$(vi)$ $\Rightarrow$ $(i)$. Consider the minimal injective coresolution of the $R$-module $R^{++}$:
\begin{equation}\label{eq2.1}
\xymatrix@C=2em{0\ar[r]&R^{++}\ar[r]&E^{0}(R^{++})\ar[r]^{f^{0}}&E^{1}(R^{++})\ar[r]^{\ \ \ \ f^{1}}& \cdots \ar[r]&E^{j}(R^{++})\ar[r]^{\ \ \ \ f^{j}}&\cdots.}
\end{equation}
 Note that {\rm fd}$_R E^j(R)$ $\leq$ $n$ for all $j$ $\geq$ $0$ by hypothesis. Applying \cite[Corollary 3.2]{Huang Auscondition}, one has {\rm fd}$_R E^j(R^{++})$ $\leq$ $n$ for all $j$ $\geq$ $0$. We choose  $\Omega^{-j}(R^{++})$ to be ker$(f^j)$ for all $j$ $\geq$ $0$. Then there exists an acyclic complex of $R$-modules whose terms all have flat dimension at most $ n$, and in which every $\Omega^{-j}(R^{++})$ appears as a cycle of this complex for all $j$ $\geq$ $0$. By Lemma \ref{lem: lemactac120701}$(iv)$, $\Omega^n(\Omega^{-j}(R^{++}))$ is a cycle of a complex in $\mathbf{K}_{\rm ac}({\rm Flat}R)$ for all $j$ $\geq$ $0$. In particular, $\Omega^n(\Omega^{-(n+1)}(R^{++}))$ is Gorenstein flat by $(vi)$, and hence the Gorenstein flat dimension of $\Omega^{-(n+1)}(R^{++})$ is at most $n$. Since $R$ is right Noetherian, $R^{++}$ is a flat $R$-module.
By the argument established above, each term in the  minimal injective coresolution {\rm (2.1)} has finite flat dimension. This shows fd$_R\Omega^{-(n+1)}(R^{++})$ $<$ $\infty$, and therefore fd$_R\Omega^{-(n+1)}(R^{++})$ $\leq$ $n$ by \cite[Theorem 2.2]{Bennis2011}. Dimension shifting shows that Ext$^{n+1}_R(\Omega^{-(n+1)}(R^{++}),R^{++})$ $=$ $0$. It follows that ${\rm Ext}^1_R(\Omega^{-(n+1)}(R^{++}),\Omega^{-n}(R^{++})) = 0$, and therefore the exact sequence $0\ra \Omega^{-n}(R^{++}) \ra E^{n}(R^{++})\ra \Omega^{-(n+1)}(R^{++})\ra 0$ splits and $\Omega^{-n}(R^{++})$ is injective. Hence $R^{++}$ has finite injective dimension. By \cite[Lemma 9.1.5]{EJ}, ${\rm id}_R R < \infty$, and so ${\rm silp}(R) < \infty$. By assumption, conditions (2) and (3) hold. Thus Corollary \ref{cor:corcothac1} implies ${\rm silp}(R) = {\rm spli}(R)$, and therefore ${\rm spli}(R) < \infty$. Since ${\rm sfli}(R) \le {\rm spli}(R)$, we have ${\rm sfli}(R) < \infty$. Finally, \cite[Proposition 9.1.6]{EJ} gives ${\rm id}_{R^{\rm op}} R < \infty$, so $R$ is Iwanaga--Gorenstein.
\hfill$\Box$

\vspace{2mm}
We end this section with the following result which contains Corollary \ref{cor:1.5} in the introduction.

\begin{cor}\label{cor:2.10}
 Let $R$ be a finite-dimensional algebra over a field with infinite dominant dimension. Then the following are equivalent:
\begin{enumerate}
\item[(i)] $R$ is self-injective.
\item[(ii)] All acyclic complexes of injective $R$-modules or $R^{op}$-modules are totally acyclic.

\item[(iii)] All acyclic complexes of projective $R$-modules or $R^{op}$-modules are totally acyclic.
\item[(iv)] All acyclic complexes of injective $R$-modules are totally acyclic.
\item[(v)] All acyclic complexes of projective $R$-modules are totally acyclic.
\item[(vi)] All acyclic complexes of injective  $R^{op}$-modules are totally acyclic.
\item[(vii)] All acyclic complexes of projective  $R^{op}$-modules are totally acyclic.
\end{enumerate}
\end{cor}
\begin{proof} The assertions $(ii)\Rightarrow(iv)$,  $(ii)\Rightarrow(vi)$, $(iii)\Rightarrow(v)$ and $(iii)\Rightarrow(vii)$ are trivial. Since $R$ is a finite-dimensional algebra, it follows that all Gorenstein projective $R$-modules or $R^{op}$-modules are exactly Gorenstein flat. By Proposition \ref{prop: prop0801}, we have $(iv)\Leftrightarrow(vii)$ and $(v)\Leftrightarrow(vi)$.

Let $0\ra R \ra E^0(R)\ra \cdots \ra E^{i}(R)\ra \cdots$ be a minimal injective coresolution of the $R$-module $R$. Since $R$ has infinite dominant dimension, all the $E^{i}(R)$ are projective. This implies that $R$ is an injective $R$-module if and only if the $R$-module $R$ has finite injective dimension. Remark that the dominant dimension of $R$ is left-right symmetric by \cite[Theorem 4]{Muller}, and the condition that $R$ is self-injective is also left-right symmetric. So the assertions $(i)\Leftrightarrow(ii)\Leftrightarrow(iii)$ follow from Theorem \ref{cor: thmeq2-1}.

$(v)\Rightarrow(i)$.  Consider the short exact sequence $0\to R\to E^0(R)\to L\to 0 $ of $R$-modules, where $E^0(R)$ is the injective envelope of $R$. Note that $R$ has infinite dominant dimension by hypothesis. It follows that $L$ is a cycle of a complex in  $\mathbf{K}_{\rm ac}({\rm Prj}R)$, and therefore $L$ is a cycle of a complex  in $\mathbf{K}_{\rm tac}({\rm Prj}R)$. Thus we have  $\Ext_{R}^{1}(L,R)=0$. It follows that  $0\to R\to E^0(R)\to L\to 0 $ is split, and so $R$ is an injective $R$-module, as desired.

$(vii)\Rightarrow(i)$. The proof is similar to that of $(v)\Rightarrow(i)$.
\end{proof}

\section{\bf Relations among totally acyclic complexes}
Recall that an $R$-module $M$ is called \emph{$FP$-injective} (or \emph{absolutely pure}) \cite{MA,S} if ${\rm Ext}_R^{1}(N,M)=0$ for all finitely presented $R$-modules $N$. Following \cite{C}, a ring  $R$ is called a \emph{right IF ring} if sfli$(R^{op})$ $=$ $0$. A ring $R$ is called an \emph{IF ring} if sfli$(R)$ $=$ sfli$(R^{op})$ $=$ $0$. In general, a right coherent and right IF ring may not be an IF ring (see \cite{C}). To discuss the relations among totally acyclic complexes, we need the following result.

\begin{prop}\label{thm:cor08020301} Let $R$ be a right coherent and right IF ring. If $\mathbf{K}_{\rm ac}({\rm Prj}R^{op})$ $=$  $\mathbf{K}_{\rm tac}({\rm Prj}R^{op})$, then $R$ is an IF ring.
\end{prop}
\begin{proof} It suffices to show that sfli$(R)$ $=$ $0$. By \cite[Theorem 3.8]{DC}, we  need to check that $R$ is $FP$-injective as an $R^{op}$-module.

Let $N$ be any finitely presented $R^{op}$-module. Note that $R$ is right IF by assumption. Thus $N$ is a submodule of a free  $R^{op}$-module by \cite[Theorem 1]{C}. Since $N$ is finitely generated, one can check that $N$ is also a submodule of a finitely generated free  $R^{op}$-module $F_0$. Then we have an exact sequence $$0\ra N\ra F_0\ra N_0\ra 0$$ of  $R^{op}$-modules. It is clear that $N_0$ is finitely presented. Repeating the procedure, we obtain an exact sequence $$0\ra N\ra F_0\ra F_{-1}\ra \cdots \ra F_{-n}\ra \cdots$$  of  $R^{op}$-modules with each $F_i$ finitely generated free. It is clear that there is an acyclic complex  $$\cdots\ra F_{n}\ra\cdots\ra F_{1}\ra F_0\ra F_{-1}\ra \cdots\ra F_{-n}\ra \cdots$$ of free  $R^{op}$-modules such that $N$ $\cong$ ker($F_{0}\ra F_{-1}$). It follows that $N$ is a cycle of a  complex in $\mathbf{K}_{\rm ac}({\rm Prj}R^{op})$.  By assumption, $\mathbf{K}_{\rm ac}({\rm Prj}R^{op})$ $=$  $\mathbf{K}_{\rm tac}({\rm Prj}R^{op})$. Then Ext$^1_R(N,R)$ $=$ $0$, and so $R$ is $FP$-injective as an  $R^{op}$-module. This completes the proof.
\end{proof}

\begin{exa} \label{em: rem2060501} Let $R$ be a right coherent and right IF ring which is not an IF ring (see \cite[Example 2]{C}). Then the following conditions hold.

$(i)$ $\mathbf{K}_{\rm ac}({\rm Prj}R^{op})$ $\neq$  $\mathbf{K}_{\rm tac}({\rm Prj}R^{op})$. Note that $R$ is a right coherent and right IF ring by assumption. If $\mathbf{K}_{\rm ac}({\rm Prj}R^{op})$ $=$  $\mathbf{K}_{\rm tac}({\rm Prj}R^{op})$, then $R$ is an IF ring by Proposition \ref{thm:cor08020301}. It is a contradiction.

$(ii)$   $\mathbf{K}_{\rm ac}({\rm Inj}R^{op})$ $=$  $\mathbf{K}_{\rm tac}({\rm Inj}R^{op})$. Note that sfli$(R^{op})$ $=$ $0$ by assumption. Then  we obtain the result by Proposition \ref{lem:lem1501}$(ii)$.

$(iii)$ $\mathbf{K}_{\rm ac}({\rm Flat}R^{op})$ $\neq$  $\mathbf{K}_{\rm tac}({\rm Flat}R^{op})$.  Assume $\mathbf{K}_{\rm ac}({\rm Flat}R^{op})$ $=$  $\mathbf{K}_{\rm tac}({\rm Flat}R^{op})$. Then we have $\mathbf{K}_{\rm ac}({\rm Prj}R^{op})$ $=$  $\mathbf{K}_{\rm tac}({\rm Prj}R^{op})$ by Proposition \ref{prop: prop0801}. By $(i)$, we have $\mathbf{K}_{\rm ac}({\rm Prj}R^{op})$ $\neq$  $\mathbf{K}_{\rm tac}({\rm Prj}R^{op})$. It is a contradiction.

$(iv)$  $\mathbf{K}_{\rm ac}({\rm Prj}R)$ $=$ $\mathbf{K}_{\rm tac}({\rm Prj}R)$. We obtain the result by $(ii)$ and Proposition \ref{prop: prop0801}.

$(v)$ $\mathbf{K}_{\rm ac}({\rm Inj}R)$ $\neq$  $\mathbf{K}_{\rm tac}({\rm Inj}R)$.  Assume $\mathbf{K}_{\rm ac}({\rm Inj}R)$ $=$  $\mathbf{K}_{\rm tac}({\rm Inj}R)$.  Then $\mathbf{K}_{\rm ac}({\rm Prj}R^{op})$ $=$  $\mathbf{K}_{\rm tac}({\rm Prj}R^{op})$  by Proposition \ref{prop: prop0801}. Note that $\mathbf{K}_{\rm ac}({\rm Prj}R^{op})$ $\neq$  $\mathbf{K}_{\rm tac}({\rm Prj}R^{op})$ by $(i)$. It is a contradiction.

$(vi)$ $\mathbf{K}_{\rm ac}({\rm Flat}R)$ $=$  $\mathbf{K}_{\rm tac}({\rm Flat}R)$. This follows from $(ii)$ and Proposition \ref{prop: prop0801}.

\end{exa}

\begin{exa} \label{em: rem2060501-2} Assume that $R_1$ is a commutative coherent ring with wGgldim$(R_1)$ $=$ $\infty$, $R_2$ is not a right coherent ring with wGgldim$(R_2)$ $<$ $\infty$. Let $R$ be the direct product ring of $R_1$ and $R_2$.  Then $R$ is not right coherent by \cite[Example 3.6 3]{DB A Gflat dim-GFclosed}. Note that the character module of every injective $R^{op}$-module has finite flat dimension by \cite[Proposition 3.3]{DB A Gflat dim-GFclosed} and \cite[Theorem 5.3]{I on the fin}. Applying \cite[Theorem  3.4]{DB A Gflat dim-GFclosed}, wGgldim$(R)$ $=$ $\infty$.  Using \cite[Theorem 3.1]{DB A Gflat dim-GFclosed}, one can check that the conditions (2) and (2)$^{op}$ are equivalent for the ring $R$  if and only if the conditions (2) and (2)$^{op}$ are equivalent for the  rings $R_1$ and $R_2$. Note that {\rm wGgldim}$(R_2)$ $<$ $\infty$. By Proposition \ref{lem:lem1501}$(ii)$ and \cite[Theorem 5.3]{I on the fin}, the  conditions  (2) and (2)$^{op}$ hold for the ring $R_2$. It is clear that the two conditions (2) and (2)$^{op}$ are equivalent for the ring $R_1$ since $R_1$ is commutative.  Thus the two conditions (2) and (2)$^{op}$ are equivalent for the ring $R$, and so {\rm spli}$(R)$ $=$ {\rm silp}$(R)$  by Corollary  \ref{cor:corcothac1}.
\end{exa}

\bigskip \centerline {\bf ACKNOWLEDGEMENTS}
\bigskip
This research was partially supported by the Natural Science Foundation of the Jiangsu Higher Education Institutions of China (19KJB110012, 21KJB110003), NSFC (12201264, 12571035, 12171206, 12271481),  Jinling Institute of Technology of China (jit-fhxm-201707) and Jiangsu 333 Project. The authors would like to thank H.X. Chen, L.W. Christensen, Z.Y. Huang and J.P. Wang for helpful discussions on parts of this paper. The authors are grateful to the anonymous referee for helpful comments that improved the presentation.


\bigskip
\begin{thebibliography}{10}




%
%
%
%
\bibitem{ARk-G} M. Auslander, I. Reiten, \emph{k-Gorenstein algebras and syzygy modules}, J. Pure Appl. Algebra 92 (1994) 1-27.

\bibitem{ARS}M. Auslander, I. Reiten, S.O. Smal${\o}$, \emph{Representation Theory of Artin Algebras}, Cambridge Studies in Advanced Mathematics, Vol. 36, Cambridge University Press, Cambridge, 1995.




\bibitem{Ballas} D. Ballas, \emph{On certain homological finiteness conditions}, Comm. Algebra 45 (2017) 481-492.


\bibitem{B-O-cohe} D. Ballas, C. Chatzistavridis, \emph{On certain  homological invariants for coherent rings}, 2022, http://users.uoa.gr/~emmanoui/files/silp spli coherent.pdf.



\bibitem{Periodic} S. Bazzoni, M. Cort$\acute{e}$s-Izurdiaga, S. Estrada, \emph{Periodic modules and acyclic complexes}, Algebr. Represent. Theory 23 (2020) 1861-1883.


\bibitem{AB CM} A. Beligiannis, \emph{Cohen-Macaulay modules, (co)torsion pairs and virtually Gorenstein algebras}, J. Algebra {288} (2005) 137-211.
\bibitem{Beli and Reiten} A. Beligiannis, I. Reiten, \emph{Homological and homotopical aspects of torsion theories}, Mem. Amer. Math. Soc. 188 (2007) 1-207.
 \bibitem{DB A Gflat dim-GFclosed} D. Bennis, \emph{Rings over which the class of Gorenstein
flat modules is closed under extensions}, Comm. Algebra 37 (2009)  855-868.

\bibitem{Bennis2011} D. Bennis, \emph{A note on Gorenstein flat dimension}, Algebra Colloq. 18 (2011) 155-161.


 \bibitem{DB} D. Bennis, N. Mahdou, \emph{Global Gorenstein dimensions}, Proc. Amer. Math.
Soc.  { 138} (2010) 461-465.
\bibitem{BGGIAC}  D. Bravo, J. Gillespie, M. Hovey,  \emph{The stable module category of a general ring}, arXiv: 1405.5768v1.


\bibitem{CCLP} O. Celikbas, L.W. Christensen, L. Liang, G. Piepmeyer, \emph{Stable homology over associative rings},
Trans. Amer. Math. Soc. 369 (2017) 8061-8086.

\bibitem{ChSflat} T.J. Cheatham, D.R. Stone, \emph{Flat and projective character modules}, Proc. Amer. Math. Soc. 81 (1981) 175-177.


\bibitem{LWchrisetwG} L.W. Christensen, S. Estrada, P. Thompson, \emph{Gorenstein weak global dimension is symmetric},  Math. Nachr. 294 (2021) 2121-2128.


\bibitem{LchrFoxHol} L.W. Christensen, H.B. Foxby, H. Holm, \emph{Derived Category Methods in Commutative Algebra}, Springer Nature, Switzerland
AG 2024.

\bibitem{LWchrKatoG} L.W. Christensen, K. Kato, \emph{Totally acyclic complexes and locally Gorenstein rings}, J. Algebra Appl. {17} (2018) 1850039 (6 pages).

\bibitem{C} R.R. Colby, \emph{Rings which have flat injective modules}, J. Algebra {35} (1975) 239-252.






 \bibitem{DemEmaHom} G. Dalezios, I. Emmanouil, \emph{Homological dimension based on a class of Gorenstein flat modules},  Comptes Rendus - S$\acute{e}$rie Math$\acute{e}$matique  361  (2023) 1429-1448.
 \bibitem{DeTaon} F. Dembegioti, O. Talelli, \emph{On a relation between certain cohomological invariants}, J. Pure Appl. Algebra 212 (2008) 1432-1437.
\bibitem{DLW} Z.X. Di, L. Liang and J.P. Wang, \emph{Virtually Gorenstein rings and relative homology of
complexes}, J. Pure Appl. Algebra 227 (2023) 107127.
\bibitem{DC} N.Q. Ding, J.L. Chen, \emph{The flat dimensions of injective modules},  Manuscripta Math. {78} (1993) 165-177.
\bibitem{I on certain} I. Emmanouil, \emph{On certain cohomological invariants of groups}, Adv. Math. {225} (2010) 3446-3462.
\bibitem{I on the fin} I. Emmanouil, \emph{On the finiteness of Gorenstein homological dimensions}, J. Algebra  { 372} (2012) 376-396.
\bibitem{I and O on the flat} I. Emmanouil, O. Talelli, \emph{On the flat length of injective modules}, J. London Math. Soc. {84} (2011) 408-432.
\bibitem{I and O Total} I. Emmanouil, O. Talelli, \emph{Total acyclicity of complexes over group algebras}, arXiv: 2505.10920v1.
\bibitem{EnochMatrix} E.E. Enochs, M. Cort\'{e}s-Izurdiaga, B. Torrecillas, \emph{Gorenstein conditions over triangular matrix rings}, J. Pure Appl. Algebra { 218} (2014) 1544-1554.


 \bibitem{IcaEnocCov}  E.E. Enochs, A. Iacob, \emph{Gorenstein injective covers and envelopes over Noetherian rings}, Proc. Amer. Math. Soc. {143} (2015) 5-12.
\bibitem{enochs:gipm}E.E. Enochs,  O.M.G. Jenda, \emph{Gorenstein injective and projective modules}, \emph{Math. Z.} \textbf{220}(4)(1995) 611-633.


 \bibitem{EJ} E.E. Enochs, O.M.G. Jenda, \emph{Relative Homological Algebra}, Walter de Gruyter, Berlin-New York, 2000.
\bibitem{EJT} E.E. Enochs, O.M.G. Jenda, B. Torrecillas, \emph{Gorenstein flat modules}, Nanjing Daxue Xuebao Shuxue Bannian Kan { 10} (1993) 1-9.
 \bibitem{EHuATotal} S. Estrada, X.H. Fu, A. Iacob, \emph{Totally acyclic complexes}, J. Algebra {470} (2017) 300-319.
 \bibitem{Gecomple} T.V. Gedrich, K.W. Gruenberg, \emph{Complete cohomological functors on groups}, Topology Appl. 25 (1987) 203-223.
\bibitem{H} H. Holm, \emph{Gorenstein homological dimensions}, J. Pure Appl. Algebra { 189} (2004) 167-193.
%

\bibitem{Huang Auscondition} Z.Y. Huang, \emph{On Auslander-type conditions of modules}, Publ. Res. Inst. Math. Sci. 59 (2023) 57-88.
\bibitem{Iwer}Y. Iwanaga, \emph{On rings with finite self-injective dimension}, Comm. Algebra 7 (1979) 393-414.

\bibitem{Iwer2} Y. Iwanaga, \emph{On rings with finite self injective dimension II}, Tsukuba J. Math. 4 (1980) 107-113.


\bibitem{IyHkrAcvstac} S. Iyengar, H. Krause, \emph{Acyclicity versus total acyclicity for complexes over noetherian rings}, Doc. Math. 11 (2006) 207-240.


\bibitem{CUJensen11} C. U. Jensen, \emph{Les foncteurs d\'{e}riv\'{e}s de $\underleftarrow{lim}$ et leurs applications an th\'{e}orie des modules}, Lecture Notes in Mathematics 254, Springer, Berlin, 1972.
 \bibitem{CUJensen} C.U. Jensen, \emph{On the vanishing of $\underleftarrow{lim}^{(i)}$}, J. Algebra {15} (1970) 151-166.

\bibitem{MA} B.H. Maddox, \emph{Absolutely pure modules}, Proc. Amer. Math. Soc. {18} (1967) 155-158.


\bibitem{Muller} B. M\"{u}ller, \emph{The classification of algebras by dominant dimension}, Canad. J. Math. 20 (1968)
398-409.

 \bibitem{MurSaT} D. Murfet, S. Salarian, \emph{Totally acyclic complexes over noetherian schemes}, Adv. Math. 226 (2011)
1096-1133.

 \bibitem{Nakayama}T. Nakayama, \emph{On algebras with complete homology}, Abh. Math. Semin. Univ. Hambg 22 (1958)
300-307.



\bibitem{JJarxiv} J. $\check{\rm{S}}$aroch, J. $\check{\rm{S}}$t'ov\'{\i}$\check{\rm{c}}$ek, \emph{Singular compactness and definability for $\Sigma$-cotorsion and Gorenstein modules}, Sel. Math. New Ser. 26:23 (2020).
\bibitem{S}B. Stenstr\"{o}m, \emph{Coherent rings and $FP$-injective modules}, J. London Math. Soc. {2} (1970) 323-329.
\bibitem{Onpurityand} J. $\check{\rm{S}}$t'ov\'{\i}$\check{\rm{c}}$ek, \emph{On purity and applications to coderived and singularity categories}, arXiv:1412.1615v1.
 \bibitem{Otale1} O. Talelli, \emph{A characterization of cohomological dimension for a big class of groups}, J. Algebra 326 (2011) 238-244.


\bibitem{PJwandGYang} J.P. Wang, G. Yang, \emph{A balance result over generalized coherent rings}, Publ. Math. Debrecen {107/1-2} (2025) 139-154.






\end{thebibliography}
\end{document}